\documentclass[article, 11pt, a4]{amsart}

\usepackage{amssymb,amsfonts}
\usepackage[all,arc]{xy}
\usepackage{enumerate}
\usepackage{mathrsfs}
\usepackage{hyperref}
\usepackage[a4paper, margin=2.54cm]{geometry}
\usepackage{tikz-cd}
\usepackage{multicol}
\hypersetup{
pdftitle={Non-formality of S\^2 via the free loop space},
pdfsubject={Mathematics, Algebraic Topology, Loop Spaces},
pdfauthor={Ryan McGowan, Florian Naef and Brian O'Callaghan},
pdfkeywords={String Topology, algebraic topology,  category theory, homotopy coherent transformations, combinatorial operads, operads}}

\usepackage{bm}

\usepackage{lipsum}
\usepackage{pgfplots}

\usepackage{todonotes}

\newtheorem{thm}{Theorem}[section]
\newtheorem{cor}[thm]{Corollary}
\newtheorem{prop}[thm]{Proposition}

\theoremstyle{definition}

\newcommand{\Tot}{\mathrm{Tot}}
\newcommand{\Map}{\mathrm{Map}}
\newcommand{\Hom}{\mathrm{Hom}}
\newcommand{\cyc}{\mathrm{cyc}}
\newcommand{\op}{\mathrm{op}}
\newcommand{\ch}{\mathrm{ch}}
\newcommand{\HH}{\mathrm{HH}}

\theoremstyle{remark}
\newtheorem{rem}[thm]{Remark}

\def\header#1#2{}	
\theoremstyle{definition}

\def\xx{\times}

\newcommand{\Z}{\mathbb{Z}}

\newcommand{\oo}{\otimes}

\newcommand{\Ch}{\textbf{Ch}}

\newcommand{\sur}{\mathcal{X}}

\newcommand{\hc}{\mathrm{hc}}
\newcommand{\id}{\mathrm{id}}
\newcommand{\pt}{\mathrm{pt}}

\theoremstyle{plain}
\newtheorem{theorem}{Theorem}
\newtheorem*{theorem*}{Theorem}
\newtheorem{lemma}{Lemma}
\newtheorem{proposition}{Proposition}

\theoremstyle{definition}
\newtheorem{definition}{Definition}
\newtheorem{example}{Example}
\newtheorem*{remark}{Remark}

\makeatletter
\let\c@equation\c@thm
\makeatother
\numberwithin{equation}{section}

\title{Non-Formality of $S^2$ via the free loop space}

\address{Trinity College Dublin, Dublin, Ireland}
\email{rymcgowa@tcd.ie}
\author{Ryan McGowan}
\address{Trinity College Dublin, Dublin, Ireland}
\email{naeff@tcd.ie}
\author{Florian Naef}
\address{Trinity College Dublin, Dublin, Ireland}
\email{brocalla@tcd.ie}
\author{Brian O'Callaghan}
\date{}
 \pgfplotsset{compat=1.18}
\begin{document}

\begin{abstract}
We show that the $E_1$-equivalence $C^\bullet(S^2) \simeq H^\bullet(S^2)$ does not intertwine the inclusion of constant loops into the free loop space $S^2 \to LS^2$. That is, the isomorphism $\HH_\bullet(H^\bullet(S^2)) \cong H^\bullet(LS^2)$ does not preserve the obvious maps to $H^\bullet(S^2)$ that exist on both sides. We give an explicit computation of the defect in terms of the $E_\infty$-structure on $C^\bullet(S^2)$. Finally, we relate our calculation to recent work of Poirier-Tradler on the string topology of $S^2$.

\end{abstract}

\maketitle

\section{Introduction}\label{d1}
For a simply connected space $X$, Jones (\cite{jones1987cyclic} and corrected in \cite{Un17}) gives an isomorphism
\[
\psi \colon \HH_\bullet(N^\bullet(X)) \cong H^\bullet(LX),
\]
between the Hochschild homology of the (normalized) cochain algebra $N^\bullet(X)$, and the cohomology of the free loops space $LX = \Map(S^1, X)$. As Hochschild homology only depends on $N^\bullet(X)$ as a dg-algebra (and not as an $E_\infty$-algebra) one obtains that if $X$ is $E_1$-formal there is an isomorphism
\[
\phi \colon \HH_\bullet(H^\bullet(X)) \cong H^\bullet(LX).
\]
Since $H^\bullet(X)$ is a graded-commutative algebra there exists a natural map
\[
f \colon \HH_\bullet(H^\bullet(X)) \to H^\bullet(X).
\]
On the geometric side the inclusion of constant loop $X \to LX$ induces a map
\[
h \colon H^\bullet(LX) \to H^\bullet(X).
\]
However, as explained by \cite{Un17}, the isomorphism $\psi$ uses the $E_\infty$-structure of $N^\bullet(X)$ and thus it does not follow that $\phi$ intertwines the two maps $f$ and $h$, and indeed for $X=S^2$ this is not the case.

\begin{theorem}\label{thm:doesnotcommute}
The diagram
 \[
\begin{tikzcd}[column sep=huge, row sep=huge]
          \HH_{\bullet}(H^{\bullet}(S^2, \Z_2)) \arrow{r}{f} \arrow[d, "\phi"] & H^{\bullet}(S^2, \Z_2) \ar[d, equal] \\
        H^\bullet(LS^2, \Z_2)  \arrow{r}{h} & H^{\bullet}(S^2, \Z_2)
 \end{tikzcd}
 \]
does \emph{not} commute.
\end{theorem}

As the map $\HH_\bullet(N^\bullet(X)) \to H^\bullet(N^\bullet(X))$ can be expressed in terms of the $E_\infty$-operations, this in particular shows that $N^\bullet(S^2)$ is not $E_\infty$-formal as an $E_\infty$-algebra. Note that using factorization homology (over the inclusion of an annulus into a disk, see for instance \cite[Proposition 5.3]{ayala2015factorization}) one obtains that such an "augmentation map" only depends on the framed $E_2$-structure on $N^\bullet(X)$. We thus obtain that 
\begin{cor}
The $E_\infty$-algebra $N^\bullet(S^2)$ is not formal as a framed $E_2$-algebra (whereas it is formal as an $E_2$-algebra, see \cite{mandell09, HeutsLandPrep}).
\end{cor}

Let us briefly explain the relevance of the above to string topology. 
Evaluation on the fundamental class of $S^2$ gives a map $H^2(S^2, \Z_2) \to \Z_2$. From the above we thus get two distinct elements in $\Hom_{\Z_2}( \HH_2(H^\bullet(S^2, \Z_2)), \Z_2)$, which we denote by $\mathbf{F}$ and $\tilde{\mathbf{F}}$ for the elements induced by $f$ and $h \circ \phi$, respectively. Recall that elements in $\Hom_{\Z_2}( \HH_2(H^\bullet(S^2, \Z_2)), \Z_2)$ can be identified with (homotopy classes of) bimodule maps with higher homotopies $H^\bullet(S^2,\Z_2) \to H_\bullet(S^2, \Z_2)$ in the sense of Tradler \cite{tradler2008infinity} 

\begin{theorem}\label{thm:maintheorem2}
    The bimodule map with higher homotopies $\widetilde{\mathbf{F}}$ corresponding to $\HH_2(H^\bullet(S^2, \Z_2)) \overset{\phi}{\to} H^2(LS^2, \Z_2) \to \Z_2$ coincides with the local bimodule map with higher homotopies constructed in \cite{poirier2023note}.
\end{theorem}

Using different terminology we can interpret the above as follows. Given an $E_\infty$-algebra $A$ together with a map $A \to \Z_2[d]$ (which we call the fundamental class) there is a canonical map $\HH(A) \to A \to \Z_2[d]$ which we can ask to be a right Calabi-Yau structure. Note that in this case the fundamental class is unique, so that we can say that $E_\infty$ algebras satisfying a version of Poincar\'e duality have a canonical right Calabi-Yau structure. However, the Calabi-Yau structure does not just depend on the underlying $E_1$-algebra structure of $A$. For $A=H^\bullet(S^2, \Z_2)$ we obtain two Calabi-Yau structures, one by considering $A$ to be a commutative algebra and the other by considering $A$ as the $E_\infty$ cochain algebra on $S^2$. The main result of \cite{poirier2023note} is then that the latter of those two structures is the geometrically correct one.

\subsection{Acknowledgements:}
The first and third author would like to acknowledge the Hamilton Trust for their financial support. We thank Pavel Safronov for useful discussions and Nathalie Wahl for the initial idea of the project and for suggesting to use Ungheretti's correction.

\subsection{Conventions}
We note that $\textbf{Ch}$, the category of homologically graded chain complexes of abelian groups, is equipped with the structure of a closed monoidal symmetric category. The differential on the tensor product of two elements is given by $d(x \oo y) =dx \oo y +(-1)^{|x|}x \oo dy$. We denote the internal hom between two chain complexes $X$ and $Y$ as $\underline{\textbf{Ch}}(X,Y)$, which is itself a chain complex, that in degree $n$ consists of linear maps of degree $n$. The internal hom has differential $(d\varphi)(x) = d \circ \varphi(x) + (-1)^{|x|} \varphi(dx)$.\\

As the main calculation we wish to do is over $\Z_2$ we will denote signs arising from the Koszul rule (and others) simply by $\pm$.

\subsection{Organisation of the paper}
Section 2 recalls the Hochschild complex over which we primarily work. In section 3, we discuss the Jones-Ungheretti isomorphism, establish it in more explicit detail for computational purposes, and the homotopy-coherent natural transformations involved. Section 4 then translates this map into the language of the surjection operad of \cite{Be04}. This allows for us to demonstrate both main theorems in Section 5.

\section{Hochschild Homology of $H^\bullet(S^2, \Z_2)$}
Given an associative differential graded algebra $A$, the cyclic bar construction of $A$ is the simplicial cochain complex
\begin{align*}
    B^\cyc(A) \colon \Delta^\op &\longrightarrow \Ch \\
    n &\mapsto A^{\otimes n+1}
\end{align*}
with simplicial maps given by
\begin{align*}
\delta^j(a_0 \oo \cdots a_j \oo a_{j+1} \oo a_n) = & a_0 \oo \cdots \oo a_j a_{j+1} \oo \cdots \oo a_n \\
\delta^n(a_0 \oo \cdots \oo a_n) = & \pm a_n a_0 \oo \cdots \oo a_{n-1}\\
\sigma^i( a_0 \oo \cdots \oo a_n) = & a_0 \oo \cdots a_i  \oo 1 \oo a_{i+1} \oo \cdots \oo a_n
\end{align*}
Hochschild homology $\HH(A)$ is then defined as the homology of the (direct sum) totalization of $N^\cyc(A)$.

Let us now consider $A = N^\bullet(S^2, \Z_2)$. More precisely, we let $S^2 = \Delta^2/\partial\Delta^2$ be the simplicial two-sphere, the triangulation of the sphere consisting of one unique non-degenerate 2-simplex, one non-degenerate 0-simplex and no other non-degenerate simplices. Then we take simplicial chains on $S^2$. Note that in this case there is no differential, and we get that
\[
N^\bullet(S^2, \Z_2) = H^\bullet(S^2, \Z_2),
\]
as associative algebras with Alexander-Whitney product on the left (this is not true when considered as $E_\infty$-algebras which the main result of the paper will show).

Let $\alpha \in H^2(S^2,\Z_2)$ denote the generator of $H^2(S^2,\Z_2)$. By taking normalized chains we obtain that
\[
\HH_\bullet(H^\bullet(S^2,\Z_2)) = \bigoplus_{n \geq 0} H^{\bullet}(S^2, \Z_2) \oo \overline{H^{\bullet}(S^2, \Z_2)}^{\oo n},
\]

that is, all the differentials vanish and we obtain that $\HH_\bullet(H^\bullet(S^2,\Z_2)$ has a basis given by elements of the form $1 \oo \alpha \oo \cdots \oo \alpha$ and $\alpha \oo \alpha \oo \cdots \oo \alpha $. In particular, we obtain
\[
\HH_2(H^\bullet(S^2,\Z_2)) = \Z_2 \alpha \oplus \Z_2 (1 \otimes \alpha \otimes \alpha).
\]

\section{The Jones-Ungheretti isomophism}
We recall the construction of the isomorphism
\[
\HH_\bullet(N^\bullet(X,\Z_2)) \cong H^\bullet(LX, \Z_2),
\]
given in \cite{Un17}.

We adopt the definition of the free loop space as the following simplicial mapping space (See \cite{Lo} and \cite{jones1987cyclic} for more details). Let $S^1_\bullet = \Delta^1 / \partial \Delta^1$ be the simplicial circle. We can then form the cosimplicial space 
\[
G \colon [n] \mapsto \operatorname{Map}(S^1_n, X) = X^{n+1}
\]
which totalizes to the free loop space $LX$. Applying normalized cochains we obtain a simplicial cochain complex
\begin{align*}
G \colon \Delta^\op &\longrightarrow \Ch \\
[n] &\mapsto N^\bullet(X^{n+1}),
\end{align*}
which totalizes to cochains on $LX$ assuming $X$ is simply connected. 
\begin{theorem}[Ungheretti]
There exists a zig-zag of equivalences of simplicial cochain complexes
\[
B^\cyc(N^\bullet(X)) \xleftarrow{\sim} QB^\cyc(N^\bullet(X)) \xrightarrow{\sim} G.
\]
In particular, this induces an isomorphism
\[
\phi \colon \HH_\bullet(N^\bullet(X)) \xrightarrow{\cong} H^\bullet(LX),
\]
in case $X$ is simply-connected.
\end{theorem}
Here $QB^\cyc(N^\bullet(X))$ is the standard resolution of a simplicial chain complex which we recall below.

\subsection{Homotopy Coherent Natural Transformations}
We need to define our homotopy coherent natural transformations, which underpin the construction of the mapping $\phi$ we require. The following formulation comes from \cite{Un17}, adapted from \cite{dugger2008primer} to the context of chain complexes.

\begin{definition}
Let $I$ be a small category.
\begin{enumerate}
    \item Let $F \colon I \to \Ch$ be an $I$-diagram in $\Ch$. The standard simplicial resolution $Q_\bullet F$ of $F$ is given by 
    \[
    Q_n F(i) \colon \bigoplus_{[i_0 \to \dots \to i_n \to i] \in N_n (I/i)} F(i_0).
    \]
    We denote its geometric realization by $QF \colon I \to \Ch$.
    \item Let $F,G:I\to \textbf{Ch}$ be two diagrams in $\Ch$, we define the cosimplicial chain complex of homotopy coherent natural transformations $\text{hc}(F,G)^\bullet:\Delta\to \textbf{Ch}$ as
$$ \text{hc}(F,G)^n := \underline{\Ch}(Q_nF, G) = \prod_{\underline{\phi}\in N_n I} \underline{\textbf{Ch}}(F(i_0),G(i_n)), $$
                where the product runs over the nerve $N_\bullet I$ of $I$, that is all strings $[i_0\to \cdots \to i_n]$ of $n$-composable arrows.
    Let $\hc(F,G) \in \Ch$ be the totalization of $\hc^\bullet(F,G)$.
\end{enumerate}
\end{definition}
For every map $\psi: x \to y$ we can find a morphism $A_\psi:F(x) \to G(y)$ of degree one, and for composed morphisms $\theta:x \to y \to z$ we get a map of degree 2, $A_\theta:F(x) \to G(z)$. These maps are not necessarily chain maps, but are subject to certain coherence arguments.

\begin{example}
    Consider the category $I =\{a\rightarrow b\}$, consisting of two objects $a,b$ and one non-identity morphism: $a \to b$. Given $F,G:I\to \textbf{Ch}$, one then finds that:
                \begin{equation*}
                    \begin{split}
                        \text{hc}(F,G)^0 =& \underline{\textbf{Ch}}(F(a),G(a))\oplus \underline{\textbf{Ch}}(F(b),G(b)),\\
                        \text{hc}(F,G)^1 =& \underline{\textbf{Ch}}(F(a),G(b)),\\
                        \text{hc}(F,G)^n =&0, \quad \text{for }n\geq 2.
                    \end{split}
                \end{equation*}
    \emph{Note: }This is the normalized cosimplicial chain complex, by taking the quotient of identity morphisms, the unnormalized complex would contain additional summands in each degree.\\
    We can then find some element $\alpha \in \text{hc}(F,G)^1$ for every pair of elements $(\varphi, \psi) \in \text{hc}(F,G)^0 $ that represents the failure to commute, that is
    $$g\varphi -  \psi f = \pm d \alpha.$$
                    $$
                    \begin{tikzcd}[row sep = huge]
                        F(a)\arrow[d, "\varphi"]\arrow[r,"f"]\arrow[rd, dotted, "\alpha"]& F(b)\arrow[d, "\psi"]\\
                        G(a)\arrow[r, "g"]& G(b)
                    \end{tikzcd}
                    $$
\end{example} 

Let now $I = \Delta^\op$ and let us spell out the definition of $\hc(F,G)$ in more detail. For instance, a $0$-cycle in the chain complex $\hc(F,G)$ is given by $(A^0,A^1,\dots)$ where each $A^m = (A^m_\phi)_{\phi \in N_m(\Delta^\op)}$ is a collection of map $A^m_{[i_0 \to \dots \to i_m]} \colon F(i_0) \to G(i_m)$. They satisfy the recurrence relation
\begin{equation}
    \label{eqn:recforAs}
\sum(-1)^j\left(\delta^j A^{m-1}\right)_{\underline{\phi}}=(-1)^m d A_{\underline{\phi}}^m
\end{equation}
where $(\delta^jA^{m-1})_{\underline{\phi}}$ considers the contraction of the degree $m$ map $\phi$ into a degree $m-1$ one by ``covering up'' the $j$-th map: [$i_0 \to \cdots \to i_{j-1} \to i_{j+1} \to \cdots \to i_m$]. 
The precise structural maps are

\[\begin{gathered}
(\delta^i A)_{\underline{\phi}}= \begin{cases}
F\left(i_0\right) \xrightarrow{F(i_0 \rightarrow i_1)}F\left(i_1\right) \xrightarrow{A_{d_0 \underline{\phi}}}G\left(i_{n+1}\right) & \text { if } i=0, \\
F\left(i_0\right) \stackrel{A_{d_i \underline{\phi}}}{\longrightarrow} G\left(i_{n+1}\right) & \text { if } 0<i<n+1, \\
F\left(i_0\right) \xrightarrow{A_{d_{n+1} \underline{\phi}}} G\left(i_n\right) \xrightarrow{G(i_n \rightarrow i_{n+1})} G\left(i_{n+1}\right) & \text { if } i=n+1 .\end{cases}
\end{gathered}\]

Let us denote by $\Tot(F) \in \Ch$ the (direct sum) totalization of $F \colon \Delta^\op \to \Ch$. We then obtain a chain map
\begin{equation}\label{eqn:hcacts}
\hc(F,G) \to \underline{\Ch}(\Tot(QF), \Tot(G))).
\end{equation}
The following lemma exhibits an explicit natural section of the map $\Tot(QF) \to \Tot(F)$. Let us define (non-chain) maps
\[
M^k \colon \hc(F,G) \to \prod_{i}\underline{\Ch}(F(i+k), G(i))
\]
by the formula $(M^k)_l = \sum_{i_1, \dots, i_k} \pm A_{[\sigma_{i_1}, \dots, \sigma_{i_k}]}$ where $\sigma_i$ are the face maps in $\Delta^\op$ and 
\[
[\sigma_{i_1}, \dots, \sigma_{i_k}] = [ [l+k] \xrightarrow{\sigma_{i_1}} [l+k-1] \to \dots \to [l+1] \xrightarrow{\sigma_{i_k}} [l] ] \in N_k(\Delta^\op)
\]
\begin{lemma}\label{lm:MsfromAs}
    The maps $(M^0,M^1,\dots)$ define a chain map
    \[
    M \colon \hc(F,G) \to \underline{\Ch}(\Tot(F), \Tot(G)),
    \]
    natural in $G$ and such that the composite
    \[
    \hc(F,G) \to \underline{\Ch}(\Tot(F), \Tot(G)) \to \underline{\Ch}(\Tot(QF), \Tot(G))
    \]
    is naturally homotopic to \eqref{eqn:hcacts}.
\end{lemma}
\begin{proof}
    By definition of the differential on $\hc(F,G)$ we obtain that
    \[
    d M^k = \left(\sum_i \pm F(\sigma_i) \right) \circ M^{k-1} \pm M^{k-1} \circ \left(\sum_i \pm G(\sigma_i) \right),
    \]
    where all the intermediate terms vanish (using $(\sum_i \pm \sigma^i ) \circ (\sum_j \pm \sigma^j) = 0$). But this is exactly the differential on $\underline{\Ch}(\Tot(F), \Tot(G))$.
    
    For the second part of the statement, let us first set $G = F$ and let $\nu \in \hc(F,F)$ be the element corresponding to the augmentation $QF \to F$. Then we have $M_0(\nu) = \id$ and $M_k(\nu)=0$ for $k > 0$ and hence $\nu$ is sent to the identity. The rest follows from naturality as follows. Let $\psi \colon \Tot(F) \to \Tot(QF)$ be the image of the identity $\underline{\Ch}(QF,QF)$ under $M$. Then $M$ for a general $G$ is given by $\hc(F,G) \to \underline{\Ch}(\Tot(QF), \Tot(G))$ followed with precomposition with $\psi$. But $\psi$ is a section of $\Tot(QF) \to \Tot(F)$ by the previous observation.
\end{proof}

\begin{remark}
The above argument could be streamlined a bit using the following. Let $\ch$ be the $\Z$-linear category with objects the natural numbers and morphisms $\ch(i,i) = \Z$, $\ch(i+1,i) = \Z$ and $\ch(i,j) = 0$ otherwise. Taking unnormalized chains $C_\bullet$ can then be seen as a functor $\ch \to Z\Delta^{\op}$. With this we obtain a map
\[
C_\bullet(QF) = C_\bullet(B(\Z\Delta^\op, \Z\Delta^\op, F)) \leftarrow B(\ch, \ch, C_\bullet(F)) = Q C_\bullet(F),
\]
\end{remark}

Let $F = B^\cyc(N^\bullet(X))$ and $G = N^\bullet(\Map(S^1_\bullet,X))$ be as above.
\begin{theorem}
\emph{\textbf{[Ungheretti]}}
    There exists a map $\phi \in \mathrm{hc}(F,G)$ (canonical up to homotopy) extending the Alexander-Whitney maps $F(k-1) = N^\bullet(X)^{\oo k} \to N^\bullet(X^k) = G(k-1)$. In particular, we get an induced map
    \[
    \Tot(F) \to \Tot(G)
    \]
\end{theorem}
More precisely, for an $E_\infty$-operad $\sur$ naturally acting on $N^\bullet(-)$ extending the (associative) Alexander-Whitney product, \cite{Un17} constructs a map of cosimplicial chain complexes
\begin{equation}\label{eqn:subofhc}
\prod_{i_0 \to \dots \to i_n} \sur(i_0+1) \to \prod_{i_0 \to \dots \to i_n} \Ch(N^\bullet(X)^{i_0+1}, N^\bullet(X^{i_n+1}))
\end{equation}
given by the following construction. For $i_0 \to \dots \to i_n$ denote $\phi \colon i_0 \to i_n$ their composition and let $p = (p_1, \dots, p_{i_0+1}) \colon X^{i_n + 1} \to X^{i_0 +1 }$ be the map induced by $\Map(S^1_\bullet,X)$. We then define the map
\begin{align*}
    \sur(i_0+1) &\longrightarrow \Ch(N^\bullet(X)^{i_0+1}, N^\bullet(X^{i_n+1})) \\
    o &\mapsto \left( (a_1, \dots, a_{i_0+1}) \mapsto  \pm o(p_1^*a_1, \dots, p_{i_0+1}^* a_{i_0+1}) \right).
\end{align*}
The product of these maps over $i_0 \to \dots \to i_n$ defines \eqref{eqn:subofhc}.
Using contractibility of $\sur$ he then concludes that the Alexander-Whitney maps extend to a unique up to homotopy element on the left hand side.

\subsection{Unpacking}
Recall that our main interest is in the composition
\[
\Psi \colon \HH_2(N^\bullet(S^2)) \xrightarrow{\cong} H^2(\Tot(G)) = H^2(\Tot(N^\bullet(\Map(S^1_\bullet, S^2))) \to H^2(S^2),
\]
where the second map is induced from the map $\{ \pt \} \to S^1_\bullet$ and hence is given by the projection $\Tot(G) \to G(0)$. Moreover, we have seen that there are two classes; $1 \otimes \alpha \otimes \alpha$ and $\alpha$, both of which can be represented by cycles in the total complex of $F^{\leq 2} = \left(F(2) \to F(1) \to F(0) \right)$. It will thus suffice to unpack the restriction of the above map $M \colon \Tot(F) \to \Tot(G)$ to $\Tot(F^{\leq 2})$. That is we want to describe the maps $M$ in the diagram
$$
\begin{tikzcd}[column sep = large, row sep = large]
N^\bullet(S^2)^{\otimes 3} = F(2) \arrow[d, "\delta",swap]\arrow[r, "M^0_2"]\arrow[rd, "M^1_1"] &  G(2)\arrow[d, "\delta"] = N^\bullet(S^2 \times S^2 \times S^2)\\
N^\bullet(S^2)^{\otimes 2} = F(1) \arrow[d, "\delta",swap]\arrow[r, "M^0_1"',near start] \arrow[rd, "M^1_0",swap] &  G(1)\arrow[d, "\delta = 0"] = N^\bullet(S^2 \times S^2)\\
N^\bullet(S^2) = F(0)\arrow[r, "M^0_0 = \id",swap] &  G(0) = N^\bullet(S^2) \ar[from=uul, crossing over, "M^2_0" description, near start]
\end{tikzcd}
$$
satisfying
\begin{align*}
	dM^0 - M^0d &= 0 \\
	dM^1 \pm M^1d &= \pm \delta M^0 \pm M^0 \delta \\
	dM^2 \pm M^2d &= \pm \delta M^1 \pm M^1 \delta.
\end{align*}

Spelling out the formula in Lemma \ref{lm:MsfromAs} we obtain that $M^0 = A^0$ are given by the Alexander-Whitney products. For the other maps, let us denote by $\sigma_0, \sigma_1$ the two face maps $[1] \to [0] \in \Delta^\op([1],[0])$ and by $\nu_0, \nu_1, \nu_2$ the face maps $[1] \to [2] \in \Delta^\op([2],[1])$. We then obtain from Lemma \ref{lm:MsfromAs}
\begin{align}\label{mm}
M^1_0 &= A^1_{\sigma_0} - A^1_{\sigma_1}, \\
M^1_1 &= A^1_{\nu_0} - A^1_{\nu_1} + A^1_{\nu_2} \label{mm2} \\
M^2_0 &= \sum_{i,j} \pm A_{i,j}
\end{align}
where
\[
A_{i,j} := A^2_{[2] \overset{\nu_i}{\to} [1] \overset{\sigma_j}{\to} [0]}: N^\bullet(X)^{\otimes 3} \to N^\bullet(X).
\]

The $A^n$'s (and hence the $M^n$'s) are iteratively constructed as the (homotopically) unique solution of equation \eqref{eqn:recforAs} holding on the left hand side of \eqref{eqn:subofhc}. We now describe a particular choice of (some of) the $A^n$'s. For simplicity of notation we will (abusively) denote an element in the source of \eqref{eqn:subofhc} by its image. That abuse can alternatively be justified by the fact that \eqref{eqn:subofhc} is injective for general enough $X$.

\begin{lemma}
    We can choose
    \begin{align*}
        A^1_{\sigma_0} &= 0 \\
        A^1_{\sigma_1} &= \smile_1 \colon N^\bullet(X)^{\otimes 2} \to N^\bullet(X),
    \end{align*}
    where $\smile_1 \in \sur(2)$ is satisfying
    \[
    d(\smile_1) = \smile - \smile \circ (12),
    \]
    where $\smile \in \sur(2)$ is the associative Alexander-Whitney product and $\smile \circ (12)$ is the image under the transposition $(12)$ acting on $\sur(2)$.
\end{lemma}
\begin{proof}
    Recall from the proof in \cite{Un17} that $A^1_{\sigma_i}$ are constructed by first verifying that $\delta^i \circ M_0^0 - M_1^0 \circ \delta^i \colon N^\bullet(X)^{\otimes 2} \to N^\bullet(X)$ is induced by an element in $\sur(2)$ which is closed in $\sur(2)$. Here $\delta^i$ are the coboundary maps on $F$ and $G$, respectively. Then $A^1_{\sigma_i}$ is a primitive of that element (which exists by contractibility of $\sur(2)$).
    For $i=0$ we are getting that $\delta^0 - M_1^0 \delta^0 = 0$. More precisely, we check that the following diagram commutes \vspace{-1.8em}
    \begin{multicols}{2}
    
     \[
    \begin{tikzcd}
        N^\bullet(X)^{\otimes 2} \arrow{r}{M^0_1} \arrow{d}{\delta^0} & N^\bullet(X \times X) \arrow{d}{\Delta^*}  \\
        N^\bullet(X) \arrow{r}{=} & N^\bullet(X)
    \end{tikzcd}
    \]
    
    \[
    \begin{tikzcd}
        a\oo b \arrow{r} \arrow{d}  &\pi^*_{1}(a) \smile \pi^*_{2}(b) \arrow{d} \\
        a \smile b \arrow{r} & a\smile b
    \end{tikzcd}
    \]   
    \end{multicols}
    where $\pi^*_{i}$  denotes the pullback from the $i$-th copy of $X$ onto $X \xx \cdots \xx X$ and the map $\delta^0$ is the $0$-th boundary in the Hochschild complex, i.e. the Alexander-Whitney multiplication on $N^\bullet(X)$. For the map $\Delta^*$ we have that by naturality of the cup product
    \begin{equation*}
        \begin{split}
            \Delta^*(\pi^*_{1}(a) \smile \pi^*_{2}(b)) &= \Delta^*(\pi^*_{1})(a) \smile \Delta^*(\pi^*_{2})(b) \\
            &= (\pi_{1} \circ \Delta)^*(a) \smile (\pi_{2} \circ \Delta)^*(b) \\
            &= a \smile b.
        \end{split}
    \end{equation*}
    For the case $i=1$ a similar computation using that the $1$-st boundary on the Hochschild complex is the multiplication in opposite order shows that
    \[
    (\delta_G^1 - M_0^1 \delta_F^0)(x,y) = \pm y \smile x \pm x \smile y.
    \]
\end{proof}
\begin{remark}
    Note that $\smile_1$ exists by contractibility of $\sur(2)$. There is also a well-known formula for it (see \cite{steenrod1947products}).

\end{remark}

\begin{lemma}
    The map $M^2_0$ is induced by $V \in \sur(3)$ where $V$ satisfies
    \[
    dV = U
    \]
    and $U \in \sur(3)$ is the operation
    \[
    U(x,y,z) := \pm (x \smile y) \smile_1 z - \pm x \smile_1 (y \smile z) + \pm (z \smile x) \smile_1 y.
    \]
\end{lemma}
\begin{proof}
    Note that $d M^2_0 = -M_0^1 \delta = A^1_{\sigma^1} \delta =  \smile_1 \circ \delta$. It remains to verify that this equation holds in $\sur(3)$ and not merely in $\Ch(F(2), G(0))$. By definition we have that $M^2_0 = \sum_{i,j} \pm A_{i,j}$ and the $A_{i,j}$ are induced by elements in $\sur(3)$ which we abusively denote by the same symbols. The $A_{i,j}$ satisfy
    \[
    dA_{i,j} = \pm \delta_G^j A^1_{\nu_i} \pm A^1_{\sigma_j} \delta_F^i.
    \]
    According to the proof in \cite{Un17} the right hand side is induced by an closed element in $\sur(3)$. Let us find the corresponding element for $(i,j) = (1,1)$ and $(i,j) = (i,0)$, the other case being similar. For the first term $\delta_G^1 A^1_{\nu_1}$, let $A^1_{\nu_1}$ be induced by $o_1 \in \sur(3)$ by the formula
    \begin{align*}
        N^\bullet(X)^{\otimes 3} &\longrightarrow N^\bullet(X \times X) \\
        (a,b,c) &\mapsto o_1(p_1^*(a), p_2^*(b), p_3^*(c)).
    \end{align*}
    where $(p_1,p_2,p_3)(x,y) = (x,y,y)$. As $\delta_G^1 = \Delta^*$ we compute that
    \begin{align*}
        \Delta^* o_1(p_1^*(a), p_2^*(b), p_3^*(c)) = o_1(\Delta^*p_1^*(a), \Delta^*p_2^*(b), \Delta^*p_3^*(c)) =o_1(a,b,c),
    \end{align*}
    and hence $\delta_G^1 A^1_{\nu_1}$ is induced by $o_1$.
    For the second term we directly see (using that $A^1_{\sigma_1} = \smile_1)$ that
    \[
    A^1_{\sigma_1} \delta_F^1 (a,b,c) = a \smile_1 (b \smile c).
    \]
    We conclude that $dA_{1,1}$ is induced by the operation
    \[
    (a,b,c) \mapsto \pm o_1(a,b,c) \pm (a \smile_1 (b \smile c)).
    \]
    Likewise, we obtain that $dA_{1,0} = \pm \delta_G^0 A^1_{\nu_1}$ is induced by
    \[
    (a,b,c) \mapsto \pm o_1(a,b,c)
    \]
    where we used that $\delta_G^0 = \delta_G^1$. In particular, we observe that the $o_1$ term cancels when we sum over $j=0,1$ and we obtain the desired result.
    
\end{proof}

\section{Surjection Operad}
We now wish to solve the equation
\[
dV = U \in \sur(3),
\]
explicitly and evaluate $V(1 \otimes \alpha \otimes \alpha)$ for $X = S^2$.
We shall use the surjection operad, as defined in \cite{Be04}, as the model for $\sur$.

\begin{definition}[Surjection Operad]
    The surjection operad $\sur$ is the operad where each module $\sur(r)_d$ is generated by the non-degenerate surjections $u:\{1, \dots, r+d \} \to \{ 1, \dots, r \}$. Non-surjective maps represent the zero element in $\sur(r)_d$. Each map $u:\{1, \dots, r+d \} \to \{ 1, \dots, r \}$ can be represented by a sequence $(u(1), \dots, u(r+d))$. A surjection is degenerate when for some $1 \leq i < r+d$ we have $u(i) = u(i+1)$. \\

    \noindent The differential is given by $\delta: \sur(r)_{*} \to \sur(r)_{*-1}$
    $$\delta(u(1), \dots, u(r+d)) = \sum_{i=1}^{r+d} \pm (u(1), \dots, \widehat{u(i)}, \dots ,u(r+d))$$
    which gives the structure of a complex.

\end{definition}

\noindent Sign conventions are determined according to the table arrangement of the surjections:

\noindent The table arrangement of $u \in \sur(r)_d $ is the splitting of $u$ into subsequences 
$$u_i = (u(r_0 + \cdots + r_{i-1} + 1), u(r_0 + \cdots +r_{i-1} + 2), \dots ,u(r_0 + \cdots + r_i ))$$
such that every $u(r_0 + \cdots + r_i )$ (a \emph{caesura}) is not the last occurrence of a value $k = 1, \dots, r$, where $r_0, \dots, r_d \geq 1$ and $r_0 + \cdots + r_d = r+d$ which we then arrange into a table, which at row $i$ contains the subsequence $u_i$
\begin{equation*}
   u(1), \ldots, u(r+d)= \left\lvert \begin{array}{cccc}
u_0(1), & \cdots & u_0\left(r_0-1\right), & u_0\left(r_0\right), \\
u_1(1), & \cdots & u_1\left(r_1-1\right), & u_1\left(r_1\right), \\
\vdots & & \vdots & \vdots \\
u_{d-1}(1), & \cdots & u_{d-1}\left(r_{d-1}-1\right), & u_{d-1}\left(r_{d-1}\right), \\
u_d(1), & \cdots & u_d\left(r_d\right) .
\end{array}
\right .
\end{equation*}

We then mark the caesuras 
$$u_0(r_0), \dots ,u_i(r_i), \dots , u_d(r_d)$$ 
with alternating signs along the table, and any element in $u_d$ which has occurred beforehand is given the opposite sign to the previous occurrence. Values of $k$ which occur only once will lead to degeneracy when removed so will equal zero and require no sign. When some $u(i)$ is removed, we insert the corresponding sign from the table arrangement.

\begin{example}
    Consider the element $(4,3,1,2,1,3,5,2) \in \sur(5)_3$. The table arrangement is given by 
    \begin{equation*}
  (4,3,1,2,1,3,5,2)= \left\lvert \begin{array}{ccc}
4,& 3^+, &\\
1^-,& &\\
2^+,& &\\
1^-,& &\\
3^-,& 5,& 2^- .
\end{array}
\right .
\end{equation*}
    The differential is then given by
    \begin{align*}
        \delta(4,3,1,2,1,3,5,2)= (4,1,2,1,3,5,2) &- (4,3,2,1,3,5,2) - (4,3,1,2,3,5,2) \\
        &- (4,3,1,2,1,5,2) - (4,3,1,2,1,3,5).
    \end{align*}

\end{example}


\begin{definition}
Operadic composition is a map  $\sur(r)_d \otimes \sur(s)_e \to \sur(r+s-1)_{d+e}$ for a product $u \circ_k v \in \sur(r+s-1)_{d+e}$ by substituting occurrences of $k$ in $(u(1) \cdots u(r+d))$ with elements of $(v(1), \cdots, v(s+e))$.

Assume that $k$ has $n$ occurrences in $(u(1), \cdots ,u(r+d))$. We then consider every splitting of 
$(v(1), \cdots, v(s+e))$ into n components
$$
\left(v\left(j_0\right), \ldots, v\left(j_1\right)\right) \quad\left(v\left(j_1\right), \ldots, v\left(j_2\right)\right) \quad \cdots \quad\left(v\left(j_{n-1}\right), \ldots, v\left(j_n\right)\right)
$$
where $1=j_0 \leq j_1 \leq j_2 \leq \ldots \leq j_{r-1} \leq j_n=s+e$. Then for every value $i_m$ such that $u(i_m) = k$, replace it by the $m$-th component $\left(v\left(j_{m-1}\right) , \ldots, v\left(j_m\right) \right)$, we then increase the terms $v(j)$ by $k-1$, and terms $u(i) > k$ by $s-1$. We then take the sum of the results of all possible splittings, along with a sign. 
This sign is calculated by first giving to each component $(u_{i-1}, \dots , u_i)$ and $\left(v\left(j_{m-1}\right) , \ldots, v\left(j_m\right) \right)$ a degree, which is the amount of rows in the table arrangement it intersects, minus one. Then we calculate the sign of the permutation of 
$$(u(i_{0}), \dots ,u(i_1)) \cdots (u(i_{n}), \dots ,u(i_{n+1})) (v(j_0), \dots, v(j_1)) \cdots (v(j_{n-1}), \dots ,v(j_n)) $$
to 
$$(u(i_{0}), \cdots, u(i_1))(v(j_{0}), \cdots ,v(j_1)) \cdots  (v(j_{n-1}), \cdots, v(j_n)) (u(i_{n}), \cdots, u(i_{n+1}))  $$
which yields the sign of the composition, where the transposition of components of degree $p$ and $q$ gives sign $(-1)^{pq}$.
\end{definition}

\begin{example}
We calculate $(2,3,2,1) \circ_2 (4,3,4,1,2)$.
We have five possible splittings of $v = (4,3,4,1,2)$, which we must calculate the degree of. Using the table arrangement we get
\begin{equation*}
  (4,3,4,1,2)= \left\lvert \begin{array}{cccc}
4,&&&\\
3,&4,&1,&2.
\end{array}
\right .
\end{equation*}
We denote the degree of each term by a subscript. For instance,  $(4,3,4,1)$ intersects both lines of the table arrangement. Hence this element will have degree $1$. The five possible splittings of $v$ are given by
$$(4)_0(4, 3, 4, 1, 2)_1$$
$$(4, 3)_1(3, 4, 1, 2)_0$$
$$(4, 3, 4)_1(4, 1, 2)_0$$
$$(4, 3, 4, 1)_1(1, 2)_0$$
$$(4, 3, 4, 1, 2)_1(2)_0$$
We finally consider all permutations required. In particular, for the second splitting above we get the following permutation:
$$(2)_0(2,3,2)_1(2,1)_0(4,3)_1(3,4,1,2)_0 \mapsto (2)_0(4,3)_1(2,3,2)_1(3,4,1,2)_0(2,1)_0$$
which would add a sign of $-1$. Thus, once we consider the additions required, we get that
\begin{align*}
    (2,3,2,1) \circ_2 (4,3,4,1,2) = (5,6,5,4,5,2,3,1) &- (5,4,6,4,5,2,3,1) - (5,4,5,6,5,2,3,1 \\
    & - (5,4,5,2,6,2,3,1) - (5,4,5,2,3,6,3,1)
\end{align*}
\end{example}

\begin{rem}
    We represent the cup product, $x \smile y$ as $(1,2)$ in the surjection operad and the 1-cup product $x \smile_1 y$ as $(1,2,1)$
\end{rem}
\begin{prop}\label{surrep}
    We can represent 
    $$U = (1,3,1,2) + (1,2,3,2) - (1,2,3,1) + (3,2,3,1) + (3,1,2,1)$$
\end{prop}
\begin{proof}
    This can be calculated directly using the operadic composition for the surjection operad. For example, the first two terms are from
    $$(1,2,1) \circ_1 (1,2)$$ which is our representation of $(x \smile y) \smile_1 z$
\end{proof}

We are now ready to solve $dV = U$ for $V$ using similar ideas to \cite{mcclure2003multivariable} by appending the digit 3 to each of our terms to give a solution
\begin{prop}\label{sum}
    We can take
    $$V = (3,1,3,1,2) +(3,1,2,3,1) - (3,1,2,3,1)$$
\end{prop}
\begin{proof}
    Standard calculation of the differential verifies it satisfies the relation $dV = U$
\end{proof}

As explained in \cite{Be04} the surjection operad naturally acts on normalized chains by a generalized Alexander-Whitney formula (called the interval cut operation)
\[
AW \colon \sur(n) \otimes N_\bullet(X) \to N_\bullet(X)^{\otimes n}
\]
and hence by taking the linear dual the element $V$ defines
\[
M^2_0 \colon N^\bullet(X)^{\otimes 3} \to N^\bullet(X),
\]
which we wish to evaluate on $X = S^2$.

We compute $AW(V)$ on the unique non-degenerate 2-simplex, $e_2 \in N^2(S^2)$.
\begin{prop}
    \begin{align*}
        AW((3,1,3,1,2))(e_2) &= - e_2 \otimes e_0 \otimes e_2 \\
        AW((3,1,2,3,2))(e_2) &= + e_0 \otimes e_2 \otimes e_2 \\
        AW((3,1,2,3,1))(e_2) &= - e_2 \otimes e_0 \otimes e_2.
    \end{align*}
\end{prop}
\begin{proof}
We calculate these maps using the interval cut operation of \cite{Be04}. For example, we consider the interval cut associated to $(3,1,2,3,1)$:
\begin{figure}[h]
    \centering
    \includegraphics[width = 0.4\linewidth]{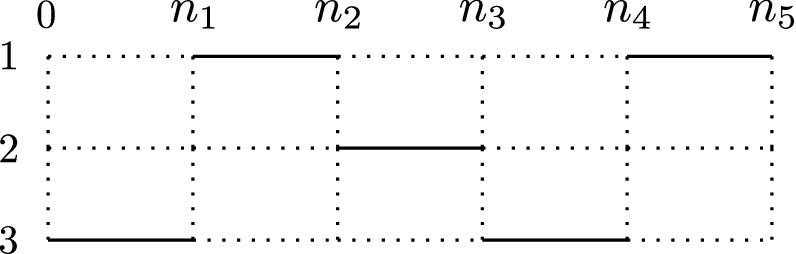}
    \caption{Interval cut of $(3,1,2,3,1)$.}
    \label{surjectionoperad}
\end{figure}

which gives 
$$\pm \Delta(n_1 \textbf{\_\hspace{-.1em}\_\hspace{-.1em}\_\hspace{-.1em}\_} n_2, n_4 \textbf{\_\hspace{-.1em}\_\hspace{-.1em}\_\hspace{-.1em}\_} n_5) \oo \Delta(n_2 \textbf{\_\hspace{-.1em}\_\hspace{-.1em}\_\hspace{-.1em}\_} n_3) \otimes \Delta(0 \textbf{\_\hspace{-.1em}\_\hspace{-.1em}\_\hspace{-.1em}\_} n_1,n_3 \textbf{\_\hspace{-.1em}\_\hspace{-.1em}\_\hspace{-.1em}\_} n_4)$$
where we then use the dimension of these simplices and the uniqueness of non-degenerate simplices to give
$$AW((3,1,2,3,1))(e_2) = \pm e_2 \otimes e_0 \otimes e_2$$
and we get the sign $(-1)$ from the permutation associated to
$\tau: (3,1,2,3,1) \mapsto (1,1,2,3,3)$
to give 
$$AW((3,1,2,3,1))(e_2) = - e_2 \otimes e_0 \otimes e_2$$
The other two computations follow similar steps.
\end{proof}

\begin{cor}
    $M^2_0(1 \otimes \alpha \otimes \alpha) = \alpha.$
\end{cor}

\section{Proof of main theorems}

\begin{proposition}
    The map $\Psi \colon \HH_2(N^\bullet(S^2)) \to N^2(S^2)$ sends
    \begin{align*}
        \Psi(1 \otimes \alpha \otimes \alpha) = \alpha.
    \end{align*}
\end{proposition}
\begin{proof}
We have seen in previous section that $\Psi(1 \otimes \alpha \otimes \alpha) = M^2_0(1 \otimes \alpha \otimes \alpha)$ which we just computed.
\end{proof}

\begin{proof}[Proof of Theorem \ref{thm:doesnotcommute}]
First note that for the simplicial $S^2$ we have that $N^\bullet(S^2) = H^\bullet(S^2)$ as algebras, which ``implements'' $E_1$-formality of $S^2$.
In the diagram
\[
\begin{tikzcd}
    \HH_\bullet(N^\bullet(S^2)) \ar[r, "f"] \ar[d, "\phi"] & H^\bullet(S^2) \ar[d, equal] \\
    H^\bullet(LS^2) \ar[r, "h"] & H^\bullet(S^2)
\end{tikzcd}
\]
we have just computed the composition $\Psi = h \circ \phi$ and shown that it is non-zero on $1 \otimes \alpha \otimes \alpha$. The map $f$ arises from the fact that $N^\bullet(S^2)$ happens to be commutative, so that Hochschild homology has an augmentation given by projecting onto the $A$-summand. In particular, it maps $1 \otimes \alpha \otimes \alpha$ to zero.
\end{proof}

To show Theorem~\ref{thm:maintheorem2} let us first recall terminology from \cite{tradler2008infinity}. An $A$-bimodule map up to homotopy is defined to be (a homotopy class of) a chain map of $A$-bimodules
\[
F \colon B(A,A,A) \otimes_A B(A,A,A) \to \Hom_{\Z_2}(A, \Z_2)
\]
where $B(A,A,A)$ is the two-sided bar resolution. Such a map has components
\[
F_{pq} \colon \overline{A}^{\otimes p} \otimes A \otimes \overline{A}^{\otimes q} \otimes A \to \Z_2,
\]
satisfying certain conditions (arising from spelling out the differential).
By tensor-hom adjunction these maps are in a one-to-one correspondence with maps $(B(A,A,A) \otimes_A B(A,A,A))_{\otimes A^e} A \to \Z_2$ where $A^e = A \otimes A^\op$. Furthermore, there is a quasi-isomorphism  
\[
(B(A,A,A) \otimes_A B(A,A,A))_{\otimes A^e} A \to B^\cyc(A)
\]
by applying the augmentation $B(A,A,A) \to A$ on the second factor. In particular, one can always choose $F$ such that $F_{pq} = 0$ for $q \neq 0$.

In \cite{poirier2023note} a homotopy inner product $\tilde{\mathbf{F}}$ on $A := H^\bullet(S^2, \Z_2)$ is constructed geometrically and described in \cite[Example 4.2]{poirier2023note}. It has non-zero components $\tilde{\mathbf{F}}_{0,0}$ as expected (given by the Poincare pairing on $H^\bullet(S^2)$) and $\tilde{\mathbf{F}}_{2,0}$ with only non-zero component
\[
\tilde{\mathbf{F}}_{2,0}(\alpha \otimes \alpha \otimes 1 \otimes 1) = 1
\]

\begin{proof}[Proof of Theorem~\ref{thm:maintheorem2}]
We have checked that the map $W \colon B^\cyc(N^\bullet(S^2)) \to N^\bullet(S^2) \to \Z_2$ given by evaluating against the fundamental class send $1 \otimes \alpha \otimes \alpha$ and $\alpha$ to $1$. The homotopy inner product $F$ arising from precomposing with the projection $(B(A,A,A) \otimes_A B(A,A,A))_{\otimes A^e} A \to B^\cyc(A)$ has non-zero components $F_{0,0}$ (easily seen to be the pairing $H^\bullet(S^2) \otimes H^\bullet(S^2) \to \Z_2$) and $F_{2,0}$. We check that
\[
F_{2,0}(\alpha \otimes \alpha \otimes 1 \otimes 1) = W(1 \otimes \alpha \otimes \alpha) = 1
\]
and all other components vanish.
\end{proof}

\bibliographystyle{alpha}
\bibliography{ref}

\end{document}